\documentclass[oneside,english,11pt]{amsart}
\usepackage[dvips]{color}
\usepackage[hypertex]{hyperref}

\usepackage[T1]{fontenc}
\usepackage[latin9]{inputenc}
\usepackage{amstext}
\usepackage{ mathrsfs }

\usepackage{amsfonts}\usepackage{amsthm}

\usepackage{amssymb}

\usepackage{graphicx}
\usepackage{enumerate}

\usepackage[english]{babel}

\numberwithin{equation}{section}
\numberwithin{figure}{section}
\theoremstyle{plain}
\newtheorem{thm}{\protect\theoremname}[section]
  \theoremstyle{plain}
  \newtheorem{prop}[thm]{\protect\propositionname}
   \theoremstyle{plain}
  \newtheorem{cor}[thm]{\protect\corollaryname}
  \theoremstyle{remark}
  
  \theoremstyle{definition}
  \newtheorem{defn}[thm]{\protect\definitionname}
\theoremstyle{definition}
  \newtheorem{exa}[thm]{\protect\examplename}
  \theoremstyle{plain}
  \newtheorem{lem}[thm]{\protect\lemmaname}

\newtheorem{maintheorem}{Theorem}

 \newtheorem{maintheoremI}{Theorem}

  \providecommand{\definitionname}{Definition}
  \providecommand{\lemmaname}{Lemma}
  \providecommand{\propositionname}{Proposition}
  \providecommand{\corollaryname}{Corollary}
  \providecommand{\remarkname}{Remark}
\providecommand{\theoremname}{Theorem}
\providecommand{\examplename}{Example}

\theoremstyle{definition}

\newcommand{\co}{\mathbb{C}}

\newcommand{\re}{\mathbb{R}}

\newcommand{\qe}{\mathbb{Q}}
\newcommand{\ze}{\mathbb{Z}}
\newcommand{\nat}{\mathbb{N}}

\newcommand{\cpt}[1]{\mathbb{C}P^{2}}

\newcommand{\F}{\mathcal{F}}
\newcommand{\Ftilde}{\tilde{\mathcal{F}}}

\newcommand{\val}{\mbox{Val}}

\newcommand{\sep}{\mbox{Sep}}
\newcommand{\iso}{\mbox{Iso}}
\newcommand{\dic}{\mbox{Dic}}
\newcommand{\cl}[1]{\mathcal{#1}}
\newcommand{\cltilde}[1]{\tilde{\mathcal{#1}}}

\newcommand{\dr}{\mbox{$\partial$}}

\begin{document}

\setcounter{section}{0}
\setcounter{thm}{0}

\author{Rog\'erio Mol \& Rudy Rosas}
\title{Differentiable equisingularity of holomorphic foliations}
\maketitle

\begin{abstract}
We prove that a  $C^{\infty}$ equivalence between germs holomorphic foliations  at $(\co^2,0)$
establishes a bijection between the sets of formal separatrices   preserving   equisingularity classes. As a consequence, if one of the foliations is of second type, so is the other and they are   equisingular.
\end{abstract}

\footnotetext[3]{Work supported by MATH-AmSud Project CNRS/CAPES/Concytec. First author supported by Universal/CNPq. Second author supported by Vicerrectorado de Investigaci\'on de la Pontificia Universidad Cat\'olica del Per\'u.}

 \medskip \medskip

 \section{Introduction}

A celebrated theorem of Zariski \cite{zariski1932} asserts that two topological equivalent germs of curves at $(\mathbb{C}^2,0)$ are necessarily equisingular, that is, their desingularization by blow-ups are combinatorially isomorphic. In \cite{camacho1984}, the authors prove the following analogous result for holomorphic foliations at $(\mathbb{C}^2,0)$, valid for the generic  class of  \emph{generalized curve} foliations:

\begin{maintheoremI}\label{equidesingularization} Let $\mathcal{F}$ and $\mathcal{F}'$ topologically equivalent germs of holomorphic foliations at $(\mathbb{C}^2,0)$. Suppose that $\mathcal{F}$ is a generalized curve. Then $\mathcal{F}'$ is also  a generalized curve foliation. Besides,   $\mathcal{F}$ and $\mathcal{F}'$ have isomorphic desingularizations.
\end{maintheoremI}

The proof of this theorem  is based upon the following result, also proved in \cite{camacho1984}:
\begin{maintheoremI}\label{s-desingularization} Let $\mathcal{F}$ be a generalized curve   foliation at $(\mathbb{C}^2,0)$ and let $\emph{Sep}(\mathcal{F})$ be its set of separatrices. Then, the desingularization of  $\emph{Sep}(\mathcal{F})$ is also the desingularization of $\mathcal{F}$.
\end{maintheoremI}
In fact, if $\mathcal{F}$ is topologically equivalent to $\mathcal{F}'$, we have that $\textrm{Sep}(\mathcal{F})$ and $\textrm{Sep}(\mathcal{F}')$ are also topological equivalent, since the separatrices of a generalized curve foliation are   convergent. Therefore, Theorem \ref{equidesingularization}  follows from Theorem \ref{s-desingularization} and Zariski's Theorem.
In general, the validity of Theorem \ref{equidesingularization} outside the class of generalized curve foliations is a difficult open problem. Actually, such a result would imply the
 topological invariance of the algebraic multiplicity of a holomorphic foliation, which is also an open problem (see \cite{rosas2009,rosas2010,rosas2016}).   The desingularization of a germ of foliation $\mathcal{F}$ is closely  related to the desingularization of its set of separatrices $\textrm{Sep}(\mathcal{F})$ --- including the purely formal ones ---, although  Theorem \ref{s-desingularization} is not always true. Another   serious difficulty  is the fact that the   topological equivalence does not naturally map   purely formal separatrices of $\mathcal{F}$ into purely   formal separatrices of $\mathcal{F}'$, as in the case of convergent separatrices.

If the equivalence between $\mathcal{F}$ and $\mathcal{F}'$ is supposed to be $C^{\infty}$, a correspondence among formal separatrices of both foliations can be established.
Let $\Phi$ be  such a $C^{\infty}$ equivalence and consider its Taylor series $\hat{\Phi}$ as a real formal diffeomorphism of $(\mathbb{C}^2,0)$.
 Let $S$ be a possibly formal separatrix of $\mathcal{F}$, which can be seen as a parametrized two-dimensional real formal surface at $(\mathbb{C}^2,0)$. Then  the formal composition $\hat{\Phi}(S)$ is a parametrized two-dimensional real formal surface at $(\mathbb{C}^2,0)$. In this setting, we have:

\begin{maintheorem}
 \label{smooth-equisingular}
  Let $\Phi$ be a $C^{\infty}$ equivalence between two germs $\mathcal{F}$ and  $\mathcal{F}'$ of singular holomorphic foliations at $(\mathbb{C}^2,0)$. Let $S$ be a separatrix of $\mathcal{F}$, considered as a parametrized two-dimensional real formal surface at $(\mathbb{C}^2,0)$. Then the following properties hold:
\begin{enumerate} \item The real formal surface $\hat{\Phi}(S)$ is a real formal reparametrization of some separatrix $S'$ of $\mathcal{F}'$, denoted by $S'=\Phi_*(S)$.
\item Let $\mathscr{S}$ be the reduced curve defined as the union of a finite collection $S_1,\ldots,S_m$ of separatrices of  $\mathcal{F}$.  Denote by $\mathscr{S}'$ the reduced curve defined as the union of $\Phi_*(S_1),\ldots,\Phi_*(S_m)$. Then $\mathscr{S}$ and $\mathscr{S}'$ are equisingular.
\end{enumerate}
\end{maintheorem}

As a consequence of Theorem \ref{smooth-equisingular}, if $\cl{F}$ and $\cl{F}'$ are $C^{\infty}$ equivalent foliations, then the sets of separatrices $\sep(\cl{F})$ and $\sep(\cl{F}')$ have isomorphic desingularizations.
Taking into account that the property described in Theorem \ref{s-desingularization} is valid for the larger class of {\em second type} foliations (see \cite{mattei2004}), we obtain the following equidesingularization result
for $C^{\infty}$ equivalent foliations:

\begin{maintheorem}
\label{equi-2nd-type-thm}
 Let $\cl{F}$ and $\cl{F}'$ be two germs of holomorphic foliations
at $(\co^{2},0)$ equivalent by a germ of $C^{\infty}$ diffeomorphism.
If $\cl{F}$ is a foliation of second type, then $\cl{F}'$ is of
second type. Moreover, $\cl{F}$ and $\cl{F}'$  are equisingular.
\end{maintheorem}

This paper is structured in the following way. In   sections \ref{foliations} and \ref{secondtype} we present basic definitions  and some
properties of second type foliations. Next, in sections \ref{pseudoanalytic} and \ref{pseudoanalyticcomplex},
we introduce the notion of characteristic curves for germs of holomorphic foliations. These are one-dimensional real curves intrinsically associated to separatrices --- both convergent and formal. Characteristic curves    are invariant by $C^{\infty}$ equivalences and
 enable us to establish a one to one correspondence among separatrices of two $C^{\infty}$ equivalent foliations. This is done in section \ref{correspondence}. Next, in section \ref{formalreal}, we introduce the
concept of formal real equivalence of formal complex curves and we show that this notion implies equisingularity
(Theorem \ref{formal-equisingular}). In section \ref{equivalence}, we present the proof of Theorem \ref{smooth-equisingular}. Finally, in section \ref{theproof}, we accomplish the proof of Theorem \ref{equi-2nd-type-thm}.

 \section{Foliations, separatrices and desingularization}
 \label{foliations}

A germ of singular holomorphic foliation $\cl{F}$ at $(\co^2,0)$  is the object defined by an equation of the form $\omega=0$, where $\omega$ is a $1-$form $\omega = P(u,v)du + Q(u,v) dv$ --- or, equivalently, by the orbits
of the germ of holomorphic vector field $
{\bf v} = - Q(u,v) \dr / \dr u + P(u,v) \dr / \dr v$ ---,
where $P,Q \in {\mathbb C}\{u,v\}$   are relatively prime, defining what we call a {\em reduced} equation.
Two reduced
$1-$forms $\omega$ and $\tilde{\omega}$ define the same foliation if and only if
$\omega = u \, \tilde{\omega}$ for some unity $u \in {\mathbb C}\{u,v\}$.
In general, we can assume that a $1-$form $\omega = P(u,v)du + Q(u,v) dv$ defines a foliation by taking as reduced equation $\omega/R =0$, where
$R = \gcd(P,Q)$.

A considerable amount  of information on the local topology and dynamics of a  foliation is given
by their {\em separatrices}.
A  separatrix for a foliation $\cl{F}$ is an invariant formal irreducible curve.
 Algebraically, it is
 defined by an irreducible formal series
$f\in {\mathbb C}[[u,v]]$,  with $f(0,0)=0$,   satisfying
$$
\omega\wedge df=fh d u\wedge d v
$$
for some formal series
$h\in {\mathbb C}[[u,v]]$.
 If   $f$ can be taken in ${\mathbb C}\{u,v\}$,  the separatrix is said to be {\em analytic} or {\em convergent}. We denote by $\sep(\cl{F})$ the set of separatrices of $\cl{F}$ at $0 \in \co^{2}$.

The singularity $0 \in \co^{2}$ for $\cl{F}$ is said to be {\em simple}      if the
 linear part ${\rm D}  {\bf v}(0)$ of a vector field ${\bf v}$ inducing $\cl{F}$  has eigenvalues $\lambda_1, \lambda_2 \in \co$ meeting   one of the following conditions:

 \par \noindent \underline{\bf Case 1:} $\lambda_1 \lambda_2 \neq 0$ and $\lambda_1 / \lambda_2 \not \in \qe^+$. We say that $0 \in \co^{2}$
is  {\em non-degenerate} or {\em complex hyperbolic}.
The set of  separatrices $\sep(\cl{F})$ is formed by two
transversal branches, both of them analytic.

 \par \noindent \underline{\bf Case 2:} $\lambda_1 \neq 0$ and  $\lambda_2= 0$. This is called a {\em saddle-node} singularity,
for which there are formal coordinates $(u,v)$ such that $\cl{F}$
is induced by

\begin{equation}
\label{saddle-node-formal}
\omega = v(1 + \lambda u^{k})du + u^{k+1} dv,
\end{equation}
where $\lambda \in \co$ and $k \in \ze_{>0}$.
The curve $\{u=0\}$, corresponding to the tangent direction defined by the non-zero eigenvalue,   defines an analytic separatrix, called {\em strong}, whereas $\{v=0\}$  is tangent to a possibly formal separatrix, called {\em weak} or {\em central}.
The integer $k+1 > 1$ is called  {\em tangency index} of $\F$ with respect to the
weak separatrix, or simply {\em weak index}, and will be denoted by 
$\text{Ind}_{0}^{w}(\F)$.

A global foliation $\cl{G}$ on a holomorphic surface $M$ corresponds to the assignment, for
 $p \in M$, of compatible local foliations $\cl{G}_{p}$. For instance, a  holomorphic
$1-$form $\omega$  on $M$ defines a foliation $\cl{G}$ by taking $\cl{G}_{p}$ as the local foliation
defined by the germification of $\omega$ at $p$.
Let $\cl{F}$ be a local foliation at $(\co^2,0)$ defined by the $1-$form $\omega$ and let
  $\pi: (M,E) \to (\co^2,0)$  be a sequence of punctual blow-ups starting at $0 \in \co^2$.
The pull-back $1-$form $\pi^*\omega$ defines a foliation
 $\Ftilde = \pi^{*} \cl{F}$ with isolated singularities on $(M,E)$  called the {\em strict transform} of $\cl{F}$ by $\pi$.
We have the definition:

\begin{defn}
Let $\cl{G}$ be a foliation on $(M,E)$, where $E$ is a normal crossings divisor. With respect to the pair
$(\cl{G},E)$, we say that   $p \in E$ is
\begin{enumerate}
\item a {\em regular} point,    if there
are local analytic coordinates $(u,v)$ at $p$ such that $E\subset \{uv=0\}$
and $\cl{G}: du = 0$;

 \item  a {\em simple  singularity},    if
  $p$ is a simple singularity for $\cl{G}$
and $E \subset \sep_{p}(\cl{G})$.

\end{enumerate}
\end{defn}

This allows us to present the notion of  reduction of singularities of a foliation with respect
to a normal crossings divisor:

\begin{defn}{ Let $\cl{G}$ be a foliation on $(M,E)$, where $E$ is a normal crossings divisor.
We say that  $(\cl{G},E)$ is {\em reduced} or {\em desingularized} if all points $p \in E$ are either  regular or simple singularities
for the pair $(\cl{G},E)$. A  {\em reduction of singularities} or {\em desingularization}  for a germ of foliation $\cl{F}$ at $(\co^{2},0)$
is a morphism $\pi: (M,E) \to
(\co^2,0)$, formed by a composition of punctual blow-ups,  such that $(\pi^{*}\cl{F},E)$
is reduced.
}\end{defn}

For   a local foliation
$\cl{F}$ at $(\co^{2},0)$,  there always exists a reduction of singularities
(see   \cite{seidenberg1968} and \cite{camacho1984}).
Besides,   there exists a {\em minimal}  one, in the sense that it factorizes,
by an additional sequence of blow-ups, any other reduction of singularities of $\cl{F}$.
In the sequel, whenever we refer to a reduction of singularities, we mean
a minimal one.

Let   $\pi: (M,E) \to
(\co^2,0)$ be a reduction of singularities for $\cl{F}$ and denote
$\cltilde{F} =\pi^{*}{F}$. The divisor  $E = \pi^{-1}(0)$
is a finite union of components which are  embedded projective lines, crossing normally at {\em corners}.
The regular points of $D$ are called {\em trace points}.
A component $D \subset E$ can be:
\begin{enumerate}
\item {\em non-dicritical}, if $D$ is $\cltilde{F}$-invariant. In this case, $D$ contains a finite number of simple singularities. Each trace singularity carries a separatrix   transversal to $E$, whose projection by $\pi$ is a branch  in $\sep(\cl{F})$.
\item {\em dicritical}, if $D$ is not $\cltilde{F}$-invariant. The definition
of desingularization gives that $D$ may intersect only non-dicritical components and that $\cltilde{F}$ is everywhere transverse do $D$. The $\pi$-image of a local leaf of $\cltilde{F}$ at each trace point of $D$ belongs to $\sep(\cl{F})$.
\end{enumerate}
For each $B \in \sep(\cl{F})$ we associate the trace point $\tau_{E}(B) \in E$ given by $\pi^*B\cap E$.
We define $\sep(D) = \{B \in  \sep(\cl{F}); \tau_{E}(B) \in D\}$ as the set of branches \emph{attached} to the
component $D \subset E$.
We thus have a decomposition $\sep (\cl{F}) = \iso (\cl{F}) \cup \dic (\cl{F})$, where
\[\iso(\cl{F}) = \bigcup_{ D \ \text{non-dicritical}} \sep(D) \qquad {\rm and} \qquad
 \dic(\cl{F}) = \bigcup_{D \ \text{dicritical}} \sep(D).\]
Separatrices in $\iso(\cl{F})$, known as {\em isolated}, can be additionally classified in two types.
A branch $B \in \iso(\cl{F})$ is \emph{strong} or of \emph{Briot and Bouquet} type if either
$\tau_{E}(B)$ is a non-degenerate singularity or  if $\tau_{E}(B)$ is a saddle-node
singularity with $\pi^*B$ as its strong separatrix. On the other hand,
$B \in \iso(\cl{F})$ is \emph{weak} if $\tau_{E}(B)$ is a saddle-node singularity whose weak separatrix is $\pi^*B$. This classification engenders the decomposition $\iso(\cl{F}) = \iso^{s}(\cl{F}) \cup \iso^{w}(\cl{F})$, where notations are self-evident. Note that $\iso(\cl{F})$
 is a finite set and all   purely
formal separatrices of $\cl{F}$ are contained in $\iso^{w}(\cl{F})$.

On the other hand, if non-empty, $\dic(\cl{F})$ is an infinite set of analytic
separatrices, called
{\em dicritical}. A foliation  $\cl{F}$ may be classified either as  {\em non-dicritical}
--- when $\sep(\cl{F})$ is finite, which happens when $\dic(\cl{F}) = \emptyset$ ---
or as {\em dicritical}, otherwise.

Let  $\F$ be a foliation at $(\co^2,0)$ with reduction of singularities $\pi: (M,E) \to (\co^2,0)$.  The {\em dual tree} associated
to  $\F$ is the acyclic, double weighted, directed graph  $\mathbb{A}^{*}(\F)$  defined in the following way:
\begin{enumerate}
\item to each component $D \subset E$ we associate a vertex $v(D)$;
\item to $v(D)$ we associate weights $n_{1}(D) \in \ze_{<0}$ and $n_{2}(D) \in \nat \cup \{\infty\}$, where
 $n_{1}(D) = D \cdot D$  is the self-intersection number of $D$ in $M$ and
 $n_{2}(D) = \# \sep(D)$;
\item there is an arrow from   $v(D_{2})$ to $v(D_{1})$   if and only if $D_{2} \cap D_{1} \neq \emptyset$ and  $D_{2}$   results from a blow-up at a point in $D_{1}$.
\end{enumerate}
The {\em valence} of a component $D \subset E$ is the number $\val(D)$ of arrows of $\mathbb{A}^{*}(\F)$ touching $v(D)$.
In other words, it is the total  number of  components of $E$ intersecting $D$ other from $D$ itself.

\begin{defn}{
Two foliations $\F$ and $\F'$ are said to be {\em equisingular} or {\em equireducible}
   if  $\mathbb{A}^{*}(\F) = \mathbb{A}^{*}(\F')$.
} \end{defn}

Let  $\F$ be a foliation at $(\co^2,0)$.
A sequence of blow-ups $\pi: (M,E) \to
(\co^2,0)$   {\em desingularizes} $\sep(\cl{F})$ if the transforms
$\pi^{*}S$ of branches $S \in \sep(\cl{F})$ are all disjoint and transverse to $E$.
We call this map,
which is supposed to be minimal, an {\em $\cl{S}$-desingularization}  or {\em $\cl{S}$-reduction} for $\F$.
Following the same procedure as in the construction of $\mathbb{A}^{*}(\F)$,  we  define the
{\em $\cl{S}$-dual tree} of $\F$, denoted as  $\mathbb{A}^{*}_{\cl{S}}(\F)$, as the dual tree associated
to the  {\em $\cl{S}$-desingularization} of $\F$. With this at hand, we have the following
 definitions:

\begin{defn}{
A germ of foliation $\F$ is {\em $\mathcal{S}$-desingularizable} or {\em $\mathcal{S}$-reducible} if $\mathbb{A}^{*}_{\cl{S}}(\F) = \mathbb{A}^{*}(\F)$, that is,
 an  $\cl{S}$-desingularization actually is a desingularization for $\F$.
}\end{defn}

\begin{defn}{
Two germs of foliations $\F$ and $\F'$ at $(\co^2,0)$ are \emph{$\cl{S}$-equisingular} or \emph{$\cl{S}$-equireducible} if $\mathbb{A}^{*}_{\cl{S}}(\F) = \mathbb{A}^{*}_{\cl{S}}(\F')$, that is,
if their sets of separatrices have equivalent desingularizations.
}\end{defn}

\section{Second type foliations}
\label{secondtype}

We keep the notation
$\pi: (M,E) \to (\co^2,0)$
for the reduction of singularities of $\cl{F}$  and
$\cltilde{F} = \pi^{*}\F$ for the strict transform foliation.
  We say that a saddle-node singularity  for
$\cltilde{F}$  is  {\em tangent} if its weak separatrix is contained in $E$.
Non-tangent saddle-nodes are also known as {\em well-oriented}. The
following definition is due to J.-F. Mattei and E. Salem (see \cite{mattei2004} and also
  \cite{cano2015}, \cite{genzmer2007}
 \cite{genzmer2016}):

\begin{defn}{  A foliation $\cl{F}$ at $(\co^{2},0)$ is of {\em second type} if there
are no tangent saddle-nodes in its reduction of singularities.
}\end{defn}

The main property of second type foliations to be used in this article is the following result,
which already appeared in
\cite[Th. 3.1.9]{mattei2004} in the non-dicritical case:
\begin{prop}
\label{s-equidesing}
Second type foliations are $\cl{S}$-desingularizable.
\end{prop}
 \begin{proof}
 We first remark that the same proof of Lemma 1 in \cite{camacho1984} applies to the following
  more general statement:   a second type foliation   with exactly two
 smooth transversal formal separatrices is simple. The result then follows
by the same arguments as in the proof of \cite[Th. 2]{camacho1984}.
\end{proof}

We establish the following   definition:
\begin{defn}
Two germs of foliations $\F$ and $\F'$   at $(\co^{2},0)$  are topologically (respectively, $C^{\infty}$) equivalent if there is
a germ of homeomorphism (respectively, $C^{\infty}$ diffeomorphism)  $\Phi: (\co^{2},0) \to (\co^{2},0) $  which sends leaves of $\F$ on leaves of $\F'$.
\end{defn}

The family of second type foliations contains the subclass of {\em generalized curve} foliations,
characterized by the absence  saddle-nodes in the
 desingularization. The property of being
a generalized curve foliation is a topological invariant and topological equivalent generalized curves
are equisingular. This is the main result in \cite{camacho1984}. Indeed, the topology of  a generalized curve foliation
  is closely related to its separatrix set,  entirely formed by
convergent curves. The aim in Theorem \ref{equi-2nd-type-thm} is to prove the equisingularity property for
the  family of second type foliations.  If all separatrices of two  second type foliations are convergent, then their topological equivalence implies  equisingularity. Actually,
there is a correspondence  between homeomorphic separatrices
for both foliations and the result follows from
 Zariski's equisingularity for curves in conjunction with the fact that a second type foliation
 is $\mathcal{S}$-desingularizable.
 However, in principle, a merely continuous equivalence map does not track   purely formal separatrices. For this reason, in the statement of Theorem \ref{equi-2nd-type-thm}, the regularity hypothesis on the equivalence map
 is strengthened and we ask for $C^{\infty}$ equivalences.

The following object was defined in \cite{genzmer2007}. A more thorough study on its
properties is found in \cite{genzmer2016}. Again, $\F$  is a germ of foliation at $(\mathbb{C}^{2},0)$
with reduction process $\pi: (M,E) \to
(\co^2,0)$.
\begin{defn}{\rm
\label{def-balanced-set}
A \emph{balanced equation of separatrices} for $\F$  is a formal meromorphic
function $\hat{F}$ whose associated  divisor is
\begin{equation}
\label{divisor-bal-eq}
 (\hat{F})_{0}-(\hat{F})_{\infty} \ = \
\sum_{S\in {\rm Iso}(\F)} (S)+ \sum_{S\in {\rm Dic}(\F)}\ a_{S} (S),
\end{equation}
where the coefficients $a_{S} \in \mathbb{Z}$  are  non-zero only for finitely many $S \in
\dic(\F)$, and, for each  dicritical  component $D \subset E$,
 the following equality holds:
\begin{equation}
\label{eq:1-balance}
\sum_{S \in {\rm Sep(D)}}a_{S} = 2- \val(D).
\end{equation}
Note that if $\F$ is non-dicritical, then
a balanced equation  is an equation for the set of
 separatrices.
}\end{defn}

We recall that the multiplicity $\rho(D)$   of a component $D \subset E$ is defined as
the algebraic multiplicity of a curve $\gamma$ at $(\mathbb{C}^{2},0)$ such that $\pi^{*} \gamma$ is transversal
to $D$  outside a corner of $E$. We have the following definition:
\begin{defn}{
 The \emph{tangency excess} of $\F$ along $E$ is the number
\begin{equation}
\label{eq-tau}
\tau_{0}(\F)=\sum_{q \in \textsl{SN}(\F) }\rho(D_{q})(\text{Ind}_{q}^{w}(\Ftilde)  -1),
\end{equation}
where $\textsl{SN}(\F) \subset E$ denotes the set of all tangent saddle-nodes, $D_q$ is the component of $E$ containing the weak separatrix of $\Ftilde$ at $q \in \textsl{SN}(\F)$
and $\text{Ind}_{q}^{w}(\Ftilde) > 1$ is the weak index.
}\end{defn}
Note that $\tau_{0}(\F) \geq 0$ and, by definition, $\tau_{0}(\F) = 0$ if and only if $\textsl{SN}(\F) = \emptyset$, that is, if and only if  $\F$ is of second type.

The {\em algebraic multiplicity} of a foliation $\F$ having   $\omega = Pdx + Qdy = 0$ as a reduced equation   is the integer $\nu_{0}(\F) = \min(\nu_{0}(P),\nu_{0}(Q))$.
The tangency excess   measures the extent that a balanced equation of separatrices computes the
 algebraic multiplicity of a foliation. This is expressed in the following fact, whose proof is found in \cite{genzmer2007}:

\begin{prop}
\label{prop:Equa-Ba} Let $\F$ be a  foliation on $(\mathbb{C}^{2},0)$ with $\hat{F}$ as a balanced equation of separatrices.   Denote by $\nu_{0}(\F)$ and $\nu_{0}(\hat{F})$ their
algebraic multiplicities. Then
\[\nu_{0}(\F)=\nu_{0}(\hat{F})-1+\tau_{0}(\F).\]
\end{prop}

We have, as a consequence:
\begin{cor}
\label{cor:Equa-Ba}
With the above notation,
\[\nu_{0}(\F)=\nu_{0}(\hat{F})-1\]
if and only if $\F$ is a second type foliation.
\end{cor}

\section{Pseudo-analytic curves}
\label{pseudoanalytic}

\begin{defn}  Consider $\gamma:[0,\epsilon)\rightarrow\mathbb{R}^k$ ($k\in\mathbb{N})$ with $\gamma(0)=0$. We say that the series
$$\hat{\gamma}=\sum\limits_{j=1}^{\infty}a_jt^j\,(a_j\in\mathbb{R}^k)$$
 is the Taylor series of $\gamma$ at $0\in\mathbb{R}$ if, for each $n\in\mathbb{N}$, there is a function  $\gamma_n:[0,\epsilon)\to \mathbb{R}^k$ with $|\gamma_n(t)|=o(t^n)$ and  such that $$\gamma(t)=\sum\limits_{j=1}^{n}a_jt^j+\gamma_n(t).$$
We say that $\hat{\gamma}$  is \emph{non-degenerate} if  $a_j\neq 0$ for some $j\in\mathbb{N}$.
\end{defn}
\begin{prop}\label{tainv} Suppose that $\gamma:[0,\epsilon)\to \mathbb{R}^k$ has a non-degenerate Taylor series $\hat{\gamma}$ at $0\in\mathbb{R}$. Let $U$  and $U'$ be neighborhoods of $0\in\mathbb{R}^k$ such that $U$ contains the image of $\gamma$. Let $\Phi:U\to U'$ be a $C^{\infty}$ diffeomorphism with $\Phi(0)=0$. Then the curve $\Phi\circ\gamma$ has a non-degenerate Taylor series at $0\in\mathbb{R}$ which is given by the formal composition $\hat{\Phi}\circ\hat{\gamma}$, where $\hat{\Phi}$ is the Taylor series of $\Phi$ at $0\in\mathbb{R}^k$.
\end{prop}

This proposition, whose proof is left to the reader, allows us to establish the following definition:

\begin{defn} Let $M$ be a $C^{\infty}$ manifold of dimension $k\in\mathbb{N}$ and consider the $C^{\infty}$ curve $\gamma:[0,\epsilon)\to M$ with $\gamma(0)=p\in M$.
 We say that $\gamma$ is \emph{pseudo-analytic} at $p\in M$ if, for some $C^{\infty}$ chart $\psi$ with $\psi(p)=0\in\mathbb{R}^k$, the curve $\psi\circ\gamma$ has a non-degenerate Taylor series at $0\in\mathbb{R}$.
\end{defn}

As a direct consequence of Proposition \ref{tainv} we have:

\begin{prop} Let $M$ and $M'$ be $C^{\infty}$ manifolds of dimension $k\in\mathbb{N}$ and let $\Phi:M\to M'$ be a $C^{\infty}$ diffeomorphism with $\Phi(p)={p'}$. Suppose that $\gamma:[0,\epsilon)\to M$ is pseudo-analytic at $p=\gamma(0)$. Then $\Phi\circ\gamma$ is pseudo-analytic at $p'\in M'$.
\end{prop}

Next we show that  the pseudo-analytic property is invariant under real blow-ups.

\begin{prop}\label{blow up}Suppose that $\gamma:[0,\epsilon)\to M$ is injective and pseudo-analytic at $p=\gamma(0)$. Let $\pi:\tilde{M}\to M$ be the  punctual real blow-up at $p\in M$. Then there exists  $\tilde{p}\in\pi^{-1}(p)$   such the curve $\tilde{\gamma}=\pi^{-1}\circ\gamma:(0,\epsilon)\to\tilde{M}$ can be continuously extended by defining $\tilde{\gamma}(0)=\tilde{p}$. Moreover, the extended curve $\tilde{\gamma}:[0,\epsilon)\to\tilde{M}$ is injective and pseudo-analytic at $\tilde{p}$. Clearly, this proposition holds if $\pi$ is any finite composition of real blow-ups at $p\in M$.
\end{prop}
\begin{proof} Take $C^{\infty}$ coordinates $(x_1,\ldots,x_k)$ at $p\in M$. Then $\gamma=(\gamma_1,\ldots,\gamma_k)$ has a Taylor series $$\left( \sum_{j=\nu}^{\infty}a_{1}^j t^j,\ldots,\sum_{j=\nu}^{\infty}a_{k}^j t^j\right)$$  with $(a_1^{\nu},\ldots,a_k^{\nu})\neq 0$. Of course we can assume that $a_1^{\nu}\neq 0$. If  $a_{2}^{\nu}\neq 0$, define the diffeomorphism $$\psi\colon (x_1,\ldots,x_k)\mapsto (x_1,x_2-\frac{a_2^{\nu}}{a_1^{\nu}}x_1,x_3,\ldots,x_k)$$ and consider  $$\psi\circ\gamma(t)=(\tilde{\gamma_1}(t), \tilde{\gamma_2}(t),\ldots,\tilde{\gamma_k}(t)).$$ Then it is easy to see that $\textrm{ord}(\hat{\tilde{\gamma_2}})>\nu$, where $\hat{\tilde{\gamma_2}}$ is the Taylor series of $\tilde{\gamma_2}$ .   Therefore, by changing coordinates if necessary we can assume that $a_1^{\nu}\neq 0$ and $a_j^{\nu}=0$ for  $j=2,\ldots,k$. Thus,  for $t>0$ small we have $$\pi^{-1}\circ\gamma(t)=\left( \gamma_1(t),\frac{\gamma_2(t)}{\gamma_1 (t)},\ldots,\frac{\gamma_k(t)}{\gamma_1(t)}\right),$$ which clearly tends to $(0,\ldots,0)$ as $t\rightarrow 0$.  Since $\gamma_1$ has a non-degenerate Taylor series, it suffices to show that
$\frac{\gamma_j(t)}{\gamma_1(t)}$ has a Taylor series for $j=2,\ldots,k$.  We will show that the formal quotient  $\sum\limits_{j=0}^{\infty}q_j t^j$  of
$\sum\limits_{j=\nu}^{\infty}a_{2}^j t^j$ by $\sum\limits_{j=\nu}^{\infty}a_{1}^j t^j$ is the Taylor series of $\frac{\gamma_2}{\gamma_1}$ at $t=0$; the other cases are equal.  Fix $n\in\mathbb{N}$. It is sufficient to show that $$R:=\frac{\gamma_2(t)}{\gamma_1(t)}-\sum\limits_{j=0}^{n}q_j t^j=o(t^n).$$ We can express $$\gamma_1(t)=\sum_{j=\nu}^{\nu+n}a_{1}^j t^j+f_1(t),$$
$$\gamma_2(t)=\sum_{j=\nu}^{\nu+n}a_{2}^j t^j +f_2(t),$$ where $f_1(t),f_2(t)=o(t^{\nu+n})$. Then
\begin{align}\frac{R}{t^n}&=\frac{\gamma_2(t)-\gamma_1(t)\sum\limits_{j=0}^{n}q_j t^j}{t^n\gamma_1(t)}\\
&=\frac{\big(\sum_{j=\nu}^{\nu+n}a_{2}^j t^j +f_2(t)\big)-\big(\sum_{j=\nu}^{\nu+n}a_{1}^j t^j+f_1(t)\big)\sum\limits_{j=0}^{n}q_j t^j}{t^n\gamma_1(t)}\\
&=\frac{o(t^{\nu+n})}{t^n\gamma_1(t)}=\frac{o(t^{\nu+n})}{t^{\nu+n}}\frac{1}{{\gamma_1(t)}/{t^{\nu}}}\rightarrow 0\textrm{ as } t\rightarrow 0.
\end{align}

\end{proof}

\section{Pseudo-analytic  curves in complex surfaces }
\label{pseudoanalyticcomplex}

Let $V$ be a complex regular surface and consider a curve $\gamma:[0,\epsilon)\to V$ pseudo-analytic at $p=\gamma(0)\in V$. In local holomorphic coordinates at $p$ the Taylor series of $\gamma$  is given as $$\hat{\gamma}=(\sum\limits_{j=1}^{\infty}a_j t^j+i\sum\limits_{j=1}^{\infty}b_j t^j, \sum\limits_{j=1}^{\infty}c_j t^j+i\sum\limits_{j=1}^{\infty}d_j t^j ),$$  where $a_j,b_j,c_j,d_j\in\mathbb{R}$. Then, if we set $\alpha_j=a_j+i b_j$ and $\beta_j=c_j+id_j$, we can write $$\hat{\gamma}=(\sum\limits_{j=1}^{\infty}\alpha_j t^j, \sum\limits_{j=1}^{\infty}\beta_j t^j ).$$ Thus, converting the real variable $t \in \re$ into a complex variable $z \in \co$,  we can view  the Taylor series of $\gamma$ as the formal complex parametrized curve
 $$(\sum\limits_{j=1}^{\infty}\alpha_j z^j, \sum\limits_{j=1}^{\infty}\beta_j z^j )$$at $p \in V$.

\begin{defn} Let $\gamma$ be a pseudo-analytic curve at $p\in V$ and let $\mathcal{C}$ be a formal parametrized complex curve at $p\in V$. We say that $\gamma$ is \emph{asymptotic to $\mathcal{C}$ at $p\in V$}    if $\hat{\gamma}$ is a formal reparametrization of $\mathcal{C}$, that is, if there exists a formal invertible complex series $\hat{\psi}=\sum\limits_{j=1}^{\infty}\sigma_j w^j$  such that $\hat{\gamma}=\mathcal{C}\circ\hat{\psi}$.
\end{defn}

\begin{prop}
\label{charac-curve-blowup}
Let $\gamma$ be an injective pseudo-analytic curve at $p\in V$, which is asymptotic to a formal complex curve $\mathcal{C}$  at $p\in V$. Let $\pi\colon \tilde{V}\to V $ be the punctual complex blow-up at $p\in V$.  Then there exists  $\tilde{p}\in\pi^{-1}(p)$   such the curve $\tilde{\gamma}=\pi^{-1}\circ\gamma:(0,\epsilon)\to\tilde{V}$ can be continuously extended by defining $\tilde{\gamma}(0)=\tilde{p}$. Moreover:
\begin{enumerate}
\item \label{item1} the extended curve $\tilde{\gamma}:[0,\epsilon)\to\tilde{V}$ is injective and pseudo-analytic at $\tilde{p}$;
\item $\tilde{\gamma}$ is asymptotic at $\tilde{p}$ to the strict transform of ${\mathcal{C}}$ by $\pi$.
\end{enumerate}
Clearly, this proposition holds if $\pi$ is any finite composition of complex blow-ups at $p\in V$.
\end{prop}
\begin{proof} Take holomorphic coordinates $(z_1,z_2)$  at $p\in V$. Then $\gamma=(\gamma_1,\gamma_2)$ has a Taylor series
$$\left( \sum_{j=\nu}^{\infty}a_{1}^j t^j,\sum_{j=\nu}^{\infty}a_{2}^j t^j\right)$$  with  $a_1^j,a_2^j\in\mathbb{C}$ and  $(a_1^{\nu},a_2^{\nu})\neq 0$. Of course we can assume that $a_1^{\nu}\neq 0$. Moreover, as in the proof of Proposition \ref{blow up}, by changing coordinates if necessary we can assume that  $a_2^{\nu}=0$. Thus,  for $t>0$ small we have $$\pi^{-1}\circ\gamma(t)=\left( \gamma_1(t),\frac{\gamma_2(t)}{\gamma_1 (t)}\right),$$ which clearly tends to $(0,0)$ as $t\rightarrow 0$.  Since $\gamma_1$ has a non-degenerate Taylor series, in order to prove item (\ref{item1}) it suffices to show that
$\frac{\gamma_2(t)}{\gamma_1(t)}$ has a Taylor series at $t=0$. In fact, proceeding as in Proposition \ref{blow up}, we  show that the Taylor series of $\frac{\gamma_2(t)}{\gamma_1(t)}$ is given by the formal quotient   of
$\sum\limits_{j=\nu}^{\infty}a_{2}^j t^j$ by $\sum\limits_{j=\nu}^{\infty}a_{1}^j t^j$ and this also implies item (2).
\end{proof}

\begin{defn}\label{characteristic} Let $\mathcal{F}$ be a  one-dimensional holomorphic foliation on a complex regular surface $V$, with a singularity
at $p\in V$. Consider  a $C^{\infty}$ curve $\gamma:[0,\epsilon)\rightarrow V$  with $\gamma(0)=0$.  We say that $\gamma$ is a \emph{characteristic curve} of $\mathcal{F}$ at $p$ if the following properties hold:
\begin{enumerate}
\item  $\gamma$ is injective and pseudo-analytic at $p\in V$;
\item  $\gamma((0,\epsilon))$ is contained in a leaf of $\mathcal{F}$.
\end{enumerate}
\end{defn}

A characteristic curve of a foliation is canonically asymptotic to a separatrix, as shown in the following result:

\begin{prop} \label{form}  Let $\mathcal{F}$ be a  one-dimensional holomorphic foliation on a complex regular surface $V$, with a singularity
at $p\in V$. If $\gamma\colon [0,\epsilon)\to V$ is a characteristic curve of $\mathcal{F}$ at $p$, then $\hat{\gamma}$ is a parametrized formal separatrix of $\mathcal{F}$.
\end{prop}
\begin{proof} Let $\omega$ be a holomorphic one form defining $\mathcal{F}$ near $p\in V$. Since $\gamma(t)$ is contained in a leaf of $\mathcal{F}$ for $t\in(0,\epsilon)$, we have that $$\omega (\gamma(t))\cdot \gamma'(t)=0 \textrm{ for all }t\in [0,\epsilon).$$ Then the Taylor series of $\omega (\gamma(t))\cdot \gamma'(t)$ at $t=0$ is null, that is,
$$\hat{\omega} (\hat{\gamma})\cdot \hat{\gamma}'=0.$$
\end{proof}

We finish this section  providing, in three examples,  an analysis of characteristic curves according to the type of separatrices to which they are asymptotic.

\begin{exa}\emph{Characteristic curves asymptotic to a dicritical separatrix.}\label{exam1}
 Let $\gamma$ be a characteristic curve asymptotic to a dicritical separatrix  $S$ of a foliation  $\mathcal{F}$. After the desingularization of $\mathcal{F}$, the strict transforms of $S$ and  $\gamma$ tend to a trace point in a dicritical component. Thus, the only possibility is that $\gamma$ is a curve contained in $S$.
\end{exa}

\begin{exa} \emph{Characteristic curves asymptotic to a strong separatrix.}\label{exam2}
 Let $\gamma$ be a characteristic curve asymptotic to a  strong separatrix  $S$ of a foliation  $\mathcal{F}$.  Let $\pi:(M,E)\to(\mathbb{C}^2,0)$ be the desingularization of $\mathcal{F}$ and denote by $\tilde{\mathcal{F}} = \pi^{*} \F$ the strict transform of $\mathcal{F}$. Then the strict transform $\tilde{S} = \pi^{*}S$   is a separatrix of some reduced singularity $p\in E$ of $\tilde{\mathcal{F}}$.  Since the separatrix $S$ is  strong, by doing one more  blow-up at $p$ if it were a saddle node,
we can assume that the singularity at $p$ is non degenerate.  Moreover, by performing  some additional blow-ups if necessary, we can assume that the ratio of eigenvalues of the singularity at $p$ has a negative real part.  Thus, we
can take holomorphic local coordinates $(u,v)$ at $p$ such that:
\begin{enumerate}
\item the foliation $\tilde{\mathcal{F}}$ at $p$ is generated by a 1-form
$$\omega=udv-vQ(u,v)du,$$ where ${Q(u,v)=\lambda+\ldots}$, $\text{Re}(\lambda)<0$.
\item $E$ is given by $\{u=0\}$;
\item $\tilde{S}$ is given by $\{v=0\}$.
\end{enumerate}
Let $\tilde{\gamma}$ be the strict transform of $\gamma$ by $\pi$. By Proposition \ref{charac-curve-blowup}, $\tilde{\gamma}$ is a characteristic curve of $\tilde{\mathcal{F}}$ asymptotic to $\tilde{S}$. We will prove that $\tilde{\gamma}$ is contained in $\tilde{S}$.  If we express $\tilde{\gamma}(t)=(u(t),v(t))$, $t\in[0,\epsilon)$,
since $\tilde{\gamma}$  is tangent to the foliation
$\tilde{\mathcal{F}}$, we have that $$u(t)v'(t)-v(t)Q\big(u(t),v(t)\big)u'(t)=0.$$ Then, if
we define $r(t)=|v(t)|^2$, a straightforward computation give us that $$r'= 2|v|^2\text{Re}\big(\frac{u'}{u}Q\big).$$
Since $\tilde{\gamma}$ has a nonzero Taylor series and $\tilde{\gamma}$ is asymptotic to $\{v=0\}$, we see   that $u(t)$
has a nonzero Taylor series $\hat{u}(t)=\sum\limits_{j\ge n}a_jt^j$, $a_j\in\mathbb{C}$, $a_n\neq 0$, $n\in\mathbb{N}$.
 From this we easily obtain that
$$\frac{u'(t)}{u(t)}=\frac{1}{t}\big(n+o(t)\big)$$ and therefore
$$r'= 2|v|^2 \frac{1}{t}\text{Re}\Big(\big(n+o(t)\big)Q\Big)$$ for all $t\in(0,\epsilon)$.
Suppose that $\tilde{\gamma}$ is not contained in $\tilde{S}$. Then  we have $|v(t)|>0$ for all $t\in(0,\epsilon)$. Thus,
since $$\text{Re}\Big(\big(n+o(t)\big)Q\Big)\rightarrow n\text{Re}(\lambda)<0,$$
for $t>0$ small enough we have that $r'(t)<0$. But this is a contradiction,  since that $r(0)=0$ and  $r(t)>0$ for  $t\in(0,\epsilon)$.
Therefore, we conclude that  a strong separatrix contains all its asymptotic characteristic curves.
\end{exa}

\begin{exa}\emph{Characteristic curves asymptotic to a weak separatrix.}
\label{normal curves}Consider a saddle-node foliation $\mathcal{F}$ at $0\in\mathbb{C}^2$  whose strong separatrix is contained  in  $\{(u,v)\colon u=0\}$. Then there is a  formal series $\hat{s}(u)=\sum_{j=1}^{\infty}c_ju^j$ such that the weak separatrix $S$ of $\mathcal{F}$ is given by $v=\hat{s}(u)$. It is known that there exists a constant $\vartheta >0$ depending only on the analytic type of the saddle-node such that, given $\eta\in\mathbb{C}^*$,   we can find a holomorphic function $f$ defined on a sector of the form $$V=\{re^{i\theta}\eta\in\mathbb{C}: 0<r<\epsilon,\, -\vartheta<\theta<\vartheta\}\,\,(\epsilon>0)$$  such that:
\begin{enumerate}
\item the graph $\{(u,f(u)):u\in V\}$ is contained in a leaf of the foliation $\mathcal{F}$;
\item \label{asint} the function $f$ has the series  $\sum\limits_{j=1}^{\infty}c_ju^j$ as asymptotic expansion at $0\in\mathbb{C}$.
\end{enumerate} Define $$\gamma(t)= ( \eta t,f(\eta t))$$ for $t\ge 0$ small and observe that the Taylor series of $\gamma$ at $t=0$ is given by $(\eta t,\hat{s}(\eta t))$. Therefore  $\gamma$ is a characteristic curve of $\mathcal{F}$ asymptotic to the separatrix $S$. We remark that the function $f$ above can be non unique, for example if $\eta$ corresponds to a ``node'' direction of the saddle-node. Thus, even if the weak separatrix is convergent, it does not contain all its  asymptotic characteristic curves. In general, the argument above can be done for a saddle-node appearing in the desingularization of a non reduced foliation  $\mathcal{F}$. Thus, we conclude that a weak separatrix $S$ of any foliation does not contain all  its  asymptotic characteristic curves, even if $S$ is convergent.
\end{exa}

\section{Correspondence of separatrices by a $C^{\infty}$ equivalence}
\label{correspondence}

The   main result of this section is the following:
\begin{thm}
\label{inv}
Let $\Phi$ be a $C^{\infty}$ equivalence between two germs $\mathcal{F}$ and $\mathcal{F}'$ of singular holomorphic foliations at $(\mathbb{C}^2,0)$. Then, given a formal separatrix $S$ of $\mathcal{F}$, there exists a unique formal separatrix $S'$ of $\mathcal{F}'$, denoted by $S'=\Phi_*(S)$, such that for any characteristic curve  $\gamma$ of $\mathcal{F}$ asymptotic to  $S$ we have that $\gamma'=\Phi(\gamma)$ is  a characteristic curve of $\mathcal{F}'$ asymptotic to  $S'$.
\end{thm}

In order to prove the theorem, we   need two preliminary  results.

\begin{lem}
\label{puiseux}
 Fix $\nu\in\mathbb{N}$.  Then, for each $j\in\mathbb{N}$ there exists a complex polynomial $P_j$ in $2j+1$ variables such that,
 if
 $$\hat{S}=(\sum\limits_{j=\nu}^{\infty}a_jz^j,\sum\limits_{j=\nu}^{\infty}b_jz^j)$$
 is a parametrized formal complex curve with $a_{\nu}=c^{\nu}$, $c\neq 0$, and we set $$\sigma_j=P_j(c,\frac{1}{c},a_{\nu+1},\ldots,a_{\nu+j-1},b_{\nu},\ldots,b_{\nu+j-1}),$$
 then $S$ can be reparametrized as
 $$(x^{\nu},\sum\limits_{j=\nu}^{\infty}\sigma_j x^j),$$
  that is, there is an invertible series $\hat{\psi}(x)$ such that $\hat{S}(\psi (x))=(x^{\nu},\sum\limits_{j=\nu}^{\infty}\sigma_j x^j)$.
\end{lem}
\begin{proof}This lemma is only a Puiseux's Parametrization putting in evidence the dependence of the final coefficients in terms of the initial ones.
\end{proof}
\begin{prop}\label{rosas} Let $\Phi$ be a $C^{\infty}$  equivalence between two  holomorphic foliations with isolated singularity at $0\in\mathbb{C}^2$ and let $J\colon\mathbb{C}^2\to\mathbb{C}^2$ be the complex conjugation. Then either $d\Phi(0)$ or $d\Phi(0)\circ J$ is a  $\mathbb{C}$-linear isomorphism of $\mathbb{C}^2 $.
\end{prop}

\begin{proof} See \cite[Lemma 4.3]{rosas2010}.

\end{proof}

\noindent\emph{Proof of Theorem \ref{inv}.}
 If $\gamma$ is a characteristic curve of $\mathcal{F}$ asymptotic to $S$, it is obvious that $\Phi (\gamma)$ is a characteristic curve of $\mathcal{F}'$. Moreover, by Proposition \ref{form} the characteristic curve $\Phi(\gamma)$ is asymptotic to the formal separatrix given by its Taylor series. The nontrivial fact is that this formal separatrix depends only on $S$ and not on the characteristic curve $\gamma$.

If $S$ is convergent we take $S'=\Phi(S)$. The proof of the theorem  will then be easy in the following cases:
\begin{enumerate}
\item $S$ is a dicritical separatrix;
\item $S$ is a strong separatrix.
\end{enumerate}
In both cases, if $\gamma$ is a characteristic curve asymptotic to $S$, by examples \ref{exam1} and \ref{exam2} we have that $\gamma\subset S$. Since  $\Phi(\gamma)$  is contained in $S'$, then $\Phi(\gamma)$ is   a characteristic curve asymptotic to  $S'$.

We begin   the proof  of the remaining case.
Let $\pi: (M,E) \to (\mathbb{C}^2,0)$ be the reduction of singularities of $\mathcal{F}$. Then, the strict transform $\tilde{S} = \pi^{*}S$   is the weak separatrix of a saddle-node singularity at some trace point $p\in E$. Clearly $p$ is not a corner and its  strong separatrix  is contained in $E$. Let $(u,v)$ be local holomorphic coordinates at $p\in E$ such that:
\begin{enumerate}
\item $p\simeq (0,0)$;
\item $E$ is given by $\{u=0\}$.
\end{enumerate} Then there exists a formal series $\hat{s}(u)=\sum_{j=1}^{\infty}c_ju^j$ such that the  $\tilde{S}$ is given by $v=\hat{s}(u)$.
Let $\gamma$ be a characteristic curve asymptotic to $S$ and let $\tilde{\gamma}(t)=\left(u(t),v(t)\right)$ be the strict transform of $\gamma$ by $\pi$.
Then   $\tilde{\gamma}$ has a non-degenerate Taylor series given by
$$(\hat{u},\hat{v})=\left(\sum\limits_{j=1}^{\infty}u_jt^j,\sum\limits_{j=1}^{\infty}v_jt^j\right),\textrm{ where } u_j,v_j\in\mathbb{C}.$$
Since $\tilde{\gamma}$ is asymptotic to $\tilde{S}$ at $p$, we deduce  that $u_1\neq 0$ and
$\hat{v}=\hat{s}\circ\hat{u}.$
Let $\hat{\pi}=\hat{\pi}_{\nu}+\hat{\pi}_{\nu+1}+\ldots$ be the Taylor series of $\pi$ at $p$, where the $\hat{\pi}_{\nu} $ is the initial part of $\pi$.   It is easy to see that the initial part of  $\hat{\tilde{\gamma}}=(\hat{u},\hat{s}\circ\hat{u})$ is given by $\hat{\tilde{\gamma}}_1=(u_1t,c_1u_1t)$. Then the initial part of $\hat{\gamma}=\hat{\pi}\circ\hat{\tilde{\gamma}}$ is $\hat{\pi}_{\nu}(u_1,c_1u_1)t^{\nu}$. Since $\Phi$ is a diffeomorphism, its initial part $\hat{\Phi}_1$ is an isomorphism, so $$\hat{\Phi}_1\circ\hat{\pi}_{\nu}(u_1,c_1u_1)\neq 0.$$
Then  the initial part of $\hat{\gamma'}=\hat{\Phi}\circ\hat{\gamma}$ is $$\hat{\gamma'}_{\nu}=\hat{\Phi}_1\circ\hat{\pi}_{\nu}(u_1,c_1u_1)t^{\nu}.$$ Thus,  we can write
\begin{equation}\hat{\gamma'}=(\sum\limits_{j=\nu}^{\infty}a_jt^j,\sum\limits_{j=\nu}^{\infty}b_jt^j),
\end{equation}
where the coefficients $a_j$ and $b_j$ are  polynomials in the coefficients of $\textrm{Re}(\hat{u})$ and $\textrm{Im}(\hat{u})$ and $(a_{\nu},b_{\nu})\neq 0$.
If $J:\mathbb{C}^2\to\mathbb{C}^2$ is the complex conjugation, by Proposition \ref{rosas} we have that either $\Phi_1$ or $\Phi_1\circ J$ is a $\mathbb{C}$-linear isomorphism.
Then there is $(a,b)\in\mathbb{C}^2\backslash\{0\}$ such that
\begin{enumerate}
\item $(a_{\nu},b_{\nu})=(au_1^{\nu},bu_1^{\nu})$, or
\item $(a_{\nu},b_{\nu})=(a\bar{u_1}^{\nu},b\bar{u_1}^{\nu})$.
\end{enumerate}
Both cases  are similar, so we only deal  with the first one.  Of course we can suppose that $a\neq 0$, so $a_{\nu}\neq 0$ for all $u_1\in\mathbb{C}^*$.
Since $\gamma'$ is a characteristic curve of $\mathcal{F}'$,
by Proposition \ref{form} the formal curve
 $$\hat{\gamma'}=(\sum\limits_{j=\nu}^{\infty}a_jz^j,\sum\limits_{j=\nu}^{\infty}b_jz^j)$$ is a parametrization of a formal separatrix $S'_{\hat{u}}$ of $\mathcal{F}'$. Moreover, $S'_{\hat{u}}$ is an isolated separatrix, otherwise $\gamma$ should be contained in a dicritical separatrix of $\mathcal{F}'$. We can apply Lemma \ref{puiseux} in order to obtain a parametrization
$$S'_{\hat{u}}=(x^{\nu},\sum\limits_{j=\nu}^{\infty}\sigma_jx^j),$$
where the coefficients $\sigma_j$ are polynomials in $1/u_1$ and in the coefficients of $\textrm{re}(\hat{u})$ and $\textrm{im}(\hat{u})$. Consider the map  $$\phi \colon \hat{u} \mapsto (\sigma_{\nu},\sigma_{\nu+1},\ldots).$$ Clearly we can identify the set of formal complex series in one variable with $\mathbb{C}^{\mathbb{N}}$.  Then the function $\phi$ is defined in some subset
$\hat{U}$ of $\mathbb{C}^*\times\mathbb{C}^{\mathbb{N}}$ and take values in the set
$$\Sigma=\left\{(\sigma_{\nu},\sigma_{\nu+1},\ldots)\in \mathbb{C}^{\mathbb{N}}\colon (x^{\nu},\sum\limits_{j=\nu}^{\infty}\sigma_jx^j)\in  \textrm{Iso}(\mathcal{F}')\right\},$$
where $ \textrm{Iso}(\mathcal{F}')$ is the set of  isolated separatrices of $\mathcal{F}'$. Observe that the set $\Sigma$ is finite and $\phi$ is continuous if we consider the product topology in $\mathbb{C}^{\mathbb{N}}$. Then it is sufficient to prove that $\hat{U}$ is connected, because in this case the map $\phi$ is constant and we can define $S'=S'_{\hat{u}}$ for any $\hat{u}\in\hat{U}$. We will prove that $\hat{U}$ is path connected.  From Example \ref{normal curves}, for any $u_1\in\mathbb{C}^*$ there exists a  characteristic curve of $\mathcal{F}$ whose strict transform by $\pi$ has a Taylor series at $p\in M$ given by $(u_1t,\hat{s}(u_1t))$. This shows that $C^*:=\mathbb{C}^*\times\{0\}^{\mathbb{N}}$ is contained in $\hat{U}$.  Since this set is path connected, it suffices to show that any $\hat{u}\in\hat{U}$ can be connected with some point in $C^*$ by a continuous path.  Fix $\hat{u}=(u_1,\ldots)\in \hat{U}$.  Then there exists a characteristic curve $\gamma$ of $\mathcal{F}$ such that its strict transform $\tilde{\gamma}=(u(t),v(t))$  by $\pi$  has $(\hat{u},\hat{s}\circ\hat{u})$ as its Taylor series at $p\in M$. We can assume that the image of $u(t)$  is contained in a sector of the form
$$V=\{re^{i\theta}\in\mathbb{C}: 0<r<\epsilon,\, a<\theta<b\}\,\, (\epsilon,a,b>0)$$ such that there exists a function $f\in\mathcal{O}(V)$ with the following properties:
\begin{enumerate}
\item the graph $\{(u,f(u)):u\in V\}$ is contained in a leaf of the foliation;
\item the function $f$ has the series  $\hat{s}=\sum\limits_{j=1}^{\infty}c_ju^j$ as asymptotic expansion at $0\in\mathbb{C}$.
\end{enumerate}
Consider  $\tilde{\gamma_0}(t)=(u(t),f(u(t)))$ and $\tilde{\gamma_1}(t)=(u_1 t,f(u_1 t))$ and observe the following:
\begin{enumerate}
\item $\tilde{\gamma_0}$ and $\tilde{\gamma_1}$ are the strict transforms by $\pi$ of  characteristic curves of $\mathcal{F}$ asymptotic to $S$;
\item $\tilde{\gamma_0}$ has $(\hat{u},\hat{s}\circ\hat{u})$ as its Taylor series;
\item If $\hat{u}_1:=(u_1,0,\ldots)\in C^*$, then $\tilde{\gamma_1}$ has $(\hat{u}_1,\hat{s}\circ\hat{u}_1)$ as its Taylor series.
\end{enumerate}
 Define the family of curves $$\Gamma_s(t)=\Big((1-s)u(t)+su_1 t,f\big((1-s)u(t)+su_1 t\big)\Big),\,\,s\in[0,1].$$ It is easy to see that
\begin{enumerate}
\item $\Gamma_0=\tilde{\gamma_0}$ and $\Gamma_1=\tilde{\gamma_1}$;
\item each $\Gamma_s$ is the strict transform of a characteristic curve asymptotic to $S$;
\item $\hat{\Gamma_s}=(\hat{u}_s,\hat{s}\circ\hat{u}_s)$, where $\hat{u}_s=u_1+\sum\limits_{j=2}^{\infty}(1-s)u_j t^j$.
\end{enumerate}
Then $\hat{u}_s$ defines a continuous path connecting $\hat{u}$ with $\hat{u}_1\in C^*$.
\qed

\section{Formal real equivalence  and equisingularity for curves}
\label{formalreal}

In this section we introduce the notion of formal real equivalence for formal complex curves at $(\mathbb{C}^2,0)$ and we prove that this notion implies the equisingularity property (Theorem \ref{formal-equisingular}).
A \emph{formal parametrized real surface} at $(\mathbb{C}^2,0)$ is a nonzero series in two variables of the form
\begin{equation}\sum\limits_{j,k\in\mathbb{N}}a_{jk}x^jy^k,\label{forsur}\end{equation} where $a_{jk}\in\mathbb{C}^2$ for all $j,k\in\mathbb{N}$. Naturally, a formal  parametrized complex curve $\sum_{j\in\mathbb{N}}\alpha_j z^j$ ($\alpha_j\in\mathbb{C}^2$) is also a formal parametrized real surface if we do the substitution $z=x+iy$. A \emph{ formal real reparametrization} of the surface \ref{forsur} is any series obtained by a substitution $(x,y)=\Psi(\bar{x},\bar{y})$, where $\Psi$ is a formal diffeomorphism of $(\mathbb{R}^2,0)$.
\begin{defn}  Let $\hat{\Phi}$ be a formal diffeomorphism of $(\mathbb{R}^4,0)$.  Let $\sigma(z)=\sum_{j\in\mathbb{N}}\sigma_j z^j$ and $\sigma'(z)=\sum_{j\in\mathbb{N}}\sigma'_j z^j$ be two formal parametrized irreducible complex curves at $(\mathbb{C}^2,0)$. We say that $\hat{\Phi}$ is a \emph{formal real equivalence} between $\sigma$ and $\sigma'$ if $\hat{\Phi}\circ\sigma$ is a formal  real reparametrization of $\sigma'$. In this situation we also say that $\sigma$ and $\sigma'$ are \emph{formally real equi\-va\-lent} by $\hat{\Phi}$.  In general, we say that $\hat{\Phi}$ is a formal real equivalence between two reduced formal complex curves  ${\mathscr{C}}$ and ${\mathscr{C}}'$ at  $(\mathbb{C}^2,0)$ if there is bijection between the irreducible components of ${\mathscr{C}}$ with the irreducible components of ${\mathscr{C}}'$ such that each pair of corresponding irreducible components are formally real equivalent by $\hat{\Phi}$.

\end{defn}

\begin{thm}
\label{formal-equisingular}
Let $\hat{\Phi}$ be a formal real equivalence  between two germs of reduced formal complex curves  ${\mathscr{C}}$ and ${\mathscr{C}}'$ at  $(\mathbb{C}^2,0)$. Then ${\mathscr{C}}$ and ${\mathscr{C}}'$ are equisingular.
\end{thm}
\begin{proof}
Let $\xi^1,\ldots,\xi^m$ and $\xi'^1,\ldots,\xi'^m$  be the irreducible components of $\mathscr{C}$ and $\mathscr{C}'$ respectively and assume that $\hat{\Phi}$ maps $\xi^k$ to $\xi'^k$ for $k=1,\ldots,m$. Let $$\sigma_k(z)=\sum\limits_{j\ge 1}a_j(k) z^j,\,\,a_j(k)\in\mathbb{C}^2$$ be a formal parametrization of   $\xi^k$.
Then $\hat{\Phi}\circ \sigma_k$ is a real formal parametrization of the irreducible component ${\xi'}^k$ of $\mathscr{C}'$, that is, there exists a real formal diffeomorphism $\psi_k\colon(\mathbb{C},0) \to (\mathbb{C},0)$ such that $\hat{\Phi}\circ \sigma_k\circ\psi_k(z)$ is a complex formal parametrization of ${\xi'}^k$. Given $n\in\mathbb{N}$, let $\xi^k_n$ and ${\xi'}^ k_{n}$ be the complex curves defined by the $n$-jets of  $\sigma_k(z)$ and $\hat{\Phi}\circ \sigma_k\circ\psi_k(z)$ respectively. Let  $\mathscr{C}_n$ and  $\mathscr{C}_n'$ be the reduced curves whose irreducible components are  $\{\xi^k_n\colon k=1,\ldots,m\}$ and $\{{\xi'}^k_n\colon k=1,\ldots,m\}$, respectively.  We know that for $n$ large enough:
\begin{enumerate}
\item $\mathscr{C}$ and $\mathscr{C}_n$  are equisingular;
\item $\mathscr{C}'$ and $\mathscr{C}'_n$  are equisingular.
\end{enumerate} Then it is sufficient to prove that the analytic curves $\mathscr{C}_n$ and  $\mathscr{C}'_n$ are topologically equivalent for $n$ large enough.
For the sake of simplicity we denote $\xi^1$, ${\xi'}^1$, $\xi^1_n$, ${\xi'}^1_n$,  $\sigma_1$  and $\psi_1$
by $\xi$, $\xi'$, $\xi_n$ , $\xi'_n$, $\sigma$ and    $\psi$, respectively. Then $\hat{\Phi}\circ \sigma\circ\psi(z)$ is a complex formal parametrization of $\xi'$ and $\xi_n$ is defined by the $n$-jet $\sigma_n$ of $\sigma$. Since the curves $\xi_n$ and $\xi'_n$ are analytic,  we will use the same notation  for the sets defined by these curves. If $\Phi_n$ is the $n$-jet of $\Phi$,
the singular real surface  $\mathscr{S}$ given by ${\Phi}_n(\xi_n)$ is asymptotic to $\hat{\Phi}\circ \sigma\circ\psi(z)$ up to order $n$. In fact, if we consider the $n$-jet $\psi_n$ of $\psi$, the real parametrization $\Phi_n\circ\sigma_n\circ\psi_n(z)$ of ${\Phi}_n(\xi_n)$ has a Taylor series coinciding with the Taylor series of $\hat{\Phi}\circ \sigma\circ\psi(z)$ up to order $n$. After a finite sequence of complex blow-ups $\pi\colon (M,E)\to(\mathbb{C}^2,0)$, the strict transform $\tilde{\xi'}$ of $\xi'$  is a regular formal curve transverse to the exceptional divisor $E$ at a point $p$. Let $(x,y)$ be holomorphic coordinates on a neighborhood  of $p$ such that:
\begin{enumerate}
\item $p\simeq (0,0)$;
\item the exceptional divisor $E$ is given by $\{x=0\}$;
\item the curve $\tilde{\xi'}$ is given by a formal equation $y=\sum\limits_{j\ge 1}c_j x^j$.
\end{enumerate}
 The following properties hold for $n$ large enough:
\begin{enumerate}
\item the strict transform $\tilde{\xi_n'}$ of $\xi_n'$ by $\pi$ intersects $E$ at the point $p$ and  is given by an analytic equation of the form $y=\zeta(x)=c_1x+o(x)$ near of $p$;
\item the strict transform $\tilde{\mathscr{S}}$ of $\mathscr{S}$ by $\pi$ intersects $E$ at the point $p$ and  is given by a $C^{\infty}$ equation of the form $y=f(x)= c_1x+o(x)$ near of $p$.
\end{enumerate} Given $\epsilon>0$, there is a set $D=\{|x|\le a, |y|\le b\}$, with $0<a,b<\epsilon$, such that
\begin{enumerate}
\item $\tilde{\xi_n'}\cap D=\{y=\zeta(x), |x|\le a\}$ and
\item $\tilde{\mathscr{S}}\cap D=\{y=f(x), |x|\le a\}$.
\end{enumerate}
We can easily construct a homeomorphism $\tilde{h}\colon D\to D$ satisfying
\begin{enumerate}
\item $\tilde{h}(\tilde{\mathscr{S}}\cap D)=\tilde{\xi_n'}\cap D$ and
\item $\tilde{h}(x,y)=(x,y)$ if $|y|=b$,
\end{enumerate}
which is extended as a homeomorphism between two neighborhoods of $E$ by setting $\tilde{h}=\textrm{id}$ outside $D$, where $\textrm{id}$ stands for the identity map. Then the map $h=\pi\circ\tilde{h}\circ\pi^{-1}$ defines a homeomorphism between two  neighborhoods $U_1$ and $U_2$ of $0\in\mathbb{C}^2$ such that:
\begin{enumerate}
\item  $h(\mathscr{S}\cap U_1)=\xi'_n\cap U_2$;
\item $h(\pi(D)) = \pi(D)$ and $h=\textrm{id}$ outside $\pi(D)$.
\end{enumerate}
Thus,   the map $\mathfrak{h}=h\circ\Phi_n$ is a topological equivalence between $\xi_n$ and $\xi_n'$. Moreover, since $\mathfrak{h}$  coincides with $\Phi_n$ outside $\pi(D)$, a similar construction as above can be successively made in an infinitesimal neighborhood of  each irreducible component of $\mathscr{C}$  in order to obtain, for $n$ large enough,  a topological equivalence $\mathfrak{h}$ between $\mathscr{C}$ and $\mathscr{C}'$.
\end{proof}

We close this section by establishing a kind of ``factorization'' theorem for a real pa\-ra\-me\-tri\-za\-tion  of an irreducible complex curve.
In more precise terms, suppose that $\xi$ is an irreducible curve at $(\co^{2},0)$
defined by the formal equation $F(u,v) = 0$. Let $$\Gamma = (f(x,y),g(x,y))$$ be a formal
parametrized
real surface at $(\co^{2},0)$ whose ``image'' is contained in $\xi$, that is, such that $F(f,g) = 0$.  Then, Lemma \ref{lemarepa} asserts that  $\Gamma$ is a formal real reparametrization of a Puiseux parametrization of $\xi$.
This result and its   Corollary \ref{cororepa} will be important in the proof of Theorem  \ref{smooth-equisingular}.

\begin{lem}
\label{lemarepa}
Let $F$ be an irreducible element in $A=\mathbb{C}[[x,y]]$ and let $\sigma\in\mathbb{C}[[z]]$, $n\in\mathbb{N}$ be such that $(z^n,\sigma(z))$ is a Puiseux parametrization for the formal curve $F=0$.  Let $f,g\in A$ be such that $F(f,g)=0$.
Then there exists a series $\psi\in A$ such that $$(f,g)=\big(\psi^n, \sigma(\psi)\big).$$
\end{lem}
\begin{proof} We will first show that it is sufficient to prove that $f$ has an $n^{\text{\tiny th}}$ root in $A$. Suppose that there exists $\phi\in A$ such that $\phi^n=f$. By Puiseux's Theorem we have that \begin{equation}\label{puiseuxfact}F(t^n,y)=U\prod\limits_{\xi^n=1}\big(y-\sigma(\xi t)\big),\end{equation}
where $U$ is a unit in $\mathbb{C}[[t,y]]$. Since $F(\phi^n,g)=F(f,g)=0$, we conclude from equation \eqref{puiseuxfact} that $g=\sigma(\xi\phi)$ for some $\xi$ such that $\xi^n=1$. Therefore it suffices to take $\psi=\xi\phi$.

Let us prove that $f$ has an $n^{\text{\tiny th}}$ root in $A$. We exclude the trivial case $n= 1$ and suppose by contradiction  that $f$ has no $n^{\text{\tiny th}}$ root in $A$. Denote by $Q$ the field of fractions of $A$.
At first we will show that, without loss of generality, we can assume that the polynomial $z^n-f$ is irreducible in $Q[z]$.
Let $d\in\mathbb{N}$ be the greatest divisor of $n$ such that $f$ has a
$d^{\text{\tiny th}}$ root in $A$. Then there exists $\tilde{f}\in A$ such that $f=\tilde{f}^d$. Since $F(\tilde{f}^{d},g)=0$, for some irreducible factor $\tilde{F}$ of $F(x^d,y)$ we have $\tilde{F}(\tilde{f},g)=0$.  If we set $\tilde{n}=\frac{n}{d}$, since $f=\tilde{f}^d$ has no $n^{\text{\tiny th}}$ root in $A$, we have that  $\tilde{f}$ has no ${\tilde{n}}^{\text{\tiny th}}$ root in $A$.  Therefore, we have that $\tilde{F}$, $\tilde{f}$ and $g$  satisfy the hypothesis of the lemma and $\tilde{f}$ has no $\tilde{n}^{\text{\tiny th}}$ root in $A$. Moreover, from the maximality of $d$ we see that, for any divisor $k\neq 1$  of $\tilde{n}$, the series  $\tilde{f}$ has no $k^{\text{\tiny th}}$ root in $A$. Thus, without loss of generality we can assume that $f$ has no $k^{\text{\tiny th}}$ root in $A$ for any divisor $k\neq 1$ of $n$. This implies that, for any  divisor $k\neq 1$ of $n$, the element $f$ has no $k^{\text{\tiny th}}$ root in the field $Q$ of fractions of $A$. From this we conclude that the polynomial $z^n-f$ is irreducible in $Q[z]$ (see, for instance, \cite[Ch. VI.9]{lang2002}).

Given any $h\in A$, define $h^*(t,x)=h(x,tx)$. If  we write $h=\sum_{j=0}^{\infty}h_j$, where $h_j$ is the homogeneous polynomial of degree $j$ in $\mathbb{C}[x,y]$, we obtain that
$$h^*(t,x)=\sum_{j=0}^{\infty} h_j(1,t)x^j.$$
Notice that $h_{j}(1,t)$ is a polynomial of degree at most $j$, so that the map $h\mapsto h^*$ defines an isomorphism from $A$ into a ring
 $A^*$ contained in the ring $K[[x]]$ of formal power series with coefficients in the field $K=\mathbb{C}(t)$ of complex rational functions in the variable $t$.
In particular,  if $f=\sum\limits_{j\ge\nu}f_j$ with $f_{\nu}\neq 0$, we obtain that $$f^*=\sum\limits_{j\ge\nu}^{\infty}f_j(1,t)x^j, \quad f_{\nu}(1,t)\neq 0.$$
 If $\bar{K}$ is the algebraic closure of $K$, we know that  the series $f^*$ has an $n^{\text{\tiny th}}$ root $\psi$ in $\bar{K}[[x]]$. Then $\psi$ is a root of the  polynomial $z^n-f^*\in A^*[z]$.
Since the polynomial $z^n-f$ is irreducible in $Q[z]$, it follows that  $n$ is the minimum degree of a nonzero polynomial in $A^*[z]$
 having $\psi$ as a root.    Since $\mathbb{C}\subset \bar{K}$, we will consider $F$ as an element in  $\bar{K}[[x,y]]$. Then, since $F(f^*,g^*)=0$ and
 $f^*=\psi^n$, it follows from Puiseux's Theorem in $\bar{K}[[x,y]]$ that $g^*=\sigma(\xi\psi)$ for some  $\xi$, $\xi^n=1$. Without loss
of generality we can assume that $\xi=1$. If we do the substitution $\psi^n=f^*$ in the equation
$g^*=\sigma(\psi)$, for some series $\sigma_1,\ldots,\sigma_{n-1}\in \mathbb{C}[[z]]$ we obtain an equation of the form
\begin{equation}\nonumber-g^*+\sigma_1(f^*)\psi+\ldots+\sigma_{n-1}(f^*)\psi^{n-1}=0.\end{equation}
 Thus, since $\sigma_j(f^*)=\big(\sigma_{j}(f)\big)^*\in A^*$, we have that $\psi$ is a root of the polynomial
\begin{equation}P=\nonumber-g^*+\sigma_1(f^*)z+\ldots+\sigma_{n-1}(f^*)z^{n-1}\in A^*[z].\end{equation}
 Then, since $n$ is the minimum degree of a polynomial in $A^*[z]$ vanishing on $\psi$, we conclude that $P=0$. Then $g=0$ and consequently we have the equation $F(f,0)=0$. Therefore, if we express $$F(x,y)=\sum\limits_{j\ge 0}s_j(x)y^j$$ with $s_j(x)\in\mathbb{C}[[x]]$, we obtain that $s_0(f)=0$. This implies that $s_0=0$, because $f\neq 0$. Then, since $F$ is irreducible, we have that $F=Uy$ for some unit $U\in A$. But this implies that $n=1$, which is a contradiction.     \end{proof}

 Let $F$ be an irreducible element in $\mathbb{C}[[x,y]]$. We say that  a  formal  parametrized complex curve
$$\Gamma(z)=\sum\limits_{j\in\mathbb{N}}a_{j}z^j,\, a_j\in\mathbb{C}^2$$ is a complex parametrization of the curve $F=0$ if $\Gamma\neq 0$ and we have $F(\Gamma(z))=0$. We say that the complex parametrization $\Gamma$ is reducible if there exist another formal  parametrized complex curve $\tilde{\Gamma}$ and an element $\varphi\in\mathbb{C}[[z]]]$ with $\textrm{ord}(\varphi)>1$ such that $\Gamma(z)=\tilde{\Gamma}(\varphi(z))$. Otherwise we say that $\Gamma$ is an irreducible complex parametrization of $F=0$. As a consequence of Lemma \ref{lemarepa}, we have:
\begin{cor}\label{cororepa}
 Let $F$ be an irreducible element in $\mathbb{C}[[x,y]]$ and let $\sigma\in\mathbb{C}[[z]]$, $n\in\mathbb{N}$ be such that $(z^n,\sigma(z))$ is a Puiseux parametrization for the formal curve $F=0$.  Let  $\Gamma$ be any irreducible complex parametrization of  $F=0$.
Then there exists a formal complex diffeomorphism $\varphi\in \mathbb{C}[[z]]$ such that $$\Gamma=\big(\varphi^n, \sigma(\varphi)\big).$$
\end{cor}
\begin{proof} Let $\Gamma=(f,g)$, where $f,g\in\mathbb{C}[[z]]$. Since $F(f,g)=0$ and $(f,g)$ can be considered as a formal real surface, by  Lemma \ref{lemarepa} there exists $\psi\in\mathbb{C}[[x,y]]$ such that  $$(f,g)=\big(\psi^n, \sigma(\psi)\big).$$ Since $(z^n,\sigma(z))$ is a Puiseux parametrization for  the curve $F=0$,  this curve is different from
 the $y$-axis and therefore $f\neq 0$. Then $\psi\neq 0$ and, since $$\psi^n=f\in\mathbb{C}[[z]],$$ we deduce that $\psi$ is in fact a non-null complex series: there exists $\varphi\in\mathbb{C}[[z]]$, $\varphi\neq 0$  such that $$\psi(x,y)=\varphi(x+iy).$$ Then we have that  $$\Gamma=\big(\varphi^n, \sigma(\varphi)\big)$$ and, since $\Gamma$ is an irreducible complex parametrization, we conclude that $\textrm{ord}(\varphi)=1$ and therefore $\varphi$ is a formal complex diffeomorphism.

\end{proof}

\section{$C^{\infty}$ equivalences of foliations and equisingularity of the set of separatrices}
\label{equivalence}

This section is devoted to prove Theorem \ref{smooth-equisingular}.

\begin{thm}\label{smooth-formal} Let $\Phi$ be a $C^{\infty}$ equivalence between two germs $\mathcal{F}$ and  $\mathcal{F}'$ of singular holomorphic foliations at $(\mathbb{C}^2,0)$. Let $S$ be a formal separatrix of $\mathcal{F}$ and let $S'=\Phi_*(S)$ be the corresponding separatrix of $\mathcal{F}'$ according to Theorem \ref{inv}.  Then $\hat{\Phi}$ is a  formal real equivalence between $S$ and $S'$.
\end{thm}

\begin{proof} Take coordinates $(z,w)$ in $(\mathbb{C}^2,0)$ an suppose that $S'$ is defined by a formal equation ${F}=0$, where $F\in\mathbb{C}[[z,w]]$ is irreducible. As a first step, considering $S$ as a formal real surface,  we will prove that  ${F}\circ\hat{\Phi}\circ{S}=0$.  Since this is obvious if $S$ is convergent, we assume that $S$ is a weak separatrix.
Let $\pi: (M,E) \rightarrow (\mathbb{C}^2,0)$ be the reduction of singularities of $\mathcal{F}$.  Then, the strict transform $\tilde{S} = \pi^{*}S$  is the weak separatrix of a saddle-node singularity at some $p\in E$.  Let $(u,v)$ be local holomorphic coordinates at $p\in E$ such that:
\begin{enumerate}
\item $p\simeq (0,0)$;
\item $E$ is given by $\{u=0\}$.
\end{enumerate} There exists a formal series $\hat{s}(u)=\sum\limits_{j=1}^{\infty}c_ju^j$ such that   $\tilde{S}$ is given by $v=\hat{s}(u)$, hence
 the separatrix $S(u)=\pi\big(u,\hat{s}(u)\big)$ is parametrized as a real surface by $${S}(x,y)={\pi}(x+iy,\hat{s}(x+iy)).$$
In order to prove that ${F}\circ\hat{\Phi}\circ{S}=0$ it is sufficient to show that, if  $\alpha,\beta\in\mathbb{R}^*$ are arbitrarily chosen, then the series  $${f}(t):={F}\circ\hat{\Phi}\circ{S}(\alpha t,\beta t)$$  is null. If we set $\eta=\alpha+i\beta$, the series ${f}$ can be expressed as
$${f}={F}\circ\hat{\Phi}\circ{S}(\eta t).$$
As we have seen in Example \ref{normal curves}, we know that ${S}(\eta t)$ is the Taylor series of a characteristic curve $\gamma$ of $\mathcal{F}$ asymptotic to $S$. Then, since $\gamma':={\Phi}(\gamma)$ is a characteristic curve of $\mathcal{F}'$ asymptotic to $S'$, we deduce that $\hat{\gamma'}=\hat{\Phi}\circ{S}(\eta t)$ is a complex formal parametrization of $S'$ and therefore
$${f}={F}\circ\hat{\Phi}\circ {S}(\eta t)={F}\circ\hat{\gamma'}=0.$$

Without loss of generality we can assume that both curves $S$ and $S'$ are tangent to the $z$ axis, which implies the following properties:
\begin{enumerate}
\item $S$ has a Puiseux parametrization $(T^n, \sigma(T))$, where $n$ is the multiplicity of the curve $S$ and  $\sigma\in\mathbb{C}[[T]]$, $\textrm{ord}(\sigma)>n$;
\item $S'$ has a Puiseux parametrization $(T^{n'}, \sigma'(T))$, where $n'$ is the multiplicity of the curve $S'$ and  $\sigma'\in\mathbb{C}[[T]]$, $\textrm{ord}(\sigma')>n'$.
\end{enumerate}
By Proposition \ref{rosas} and without loss of generality --- the other case is similar --- we can assume that $\hat{\Phi}(z,w)$ has a complex linear part \begin{equation}\label{equli} \hat{\Phi}_1(z,w)=(az+bw, cz+dw),\,ad-bc\neq 0.\end{equation}
Suppose that
\begin{equation}\label{equor} {S}(u)=\big(\sum_{j\ge \bar{n}}\alpha_j u^j,\sum_{j\ge \bar{n}}\beta_j u^j\big),\, (\alpha_{\bar{n}},\beta_{\bar{n}})\neq (0,0).\end{equation}
Since it is an irreducible parametrization, by Corollary \ref{cororepa} there exists a formal diffeomorphism $\varphi\in\mathbb{C}[[u]]$ such that ${S}=\big(\varphi^n,\sigma(\varphi)\big)$, hence $\bar{n}=n$ and $\beta_{{n}}=0$.
Therefore, from \eqref{equli} and \eqref{equor} above we have that the initial part of $\hat{\Phi}\circ{S}$ is complex and is given by
$$\big(\hat{\Phi}\circ{S}\big)_1=\big(a\alpha_{n}u^n,c\alpha_nu^n\big).$$
Since $S'$ is tangent to the $z$ axis, we have that ${F}$ has an initial part of the form $F_N=\mu y^N$, $\mu\neq 0$, $N\in\mathbb{N}$. Then, since $F\circ\hat{\Phi}\circ{S}=0$ implies $F_N\circ\big(\hat{\Phi}\circ{S}\big)_1=0$, we deduce that $c=0$ and consequently $a\neq 0$. By Lemma \ref{lemarepa}, since $F\big(\hat{\Phi}\circ{S}\big)=0$, there exists $\psi\in\mathbb{C}[[x,y]]$ such that
$$\hat{\Phi}\circ{S}(x+iy)=\big(\psi^{n'},\sigma'(\psi)\big).$$ Then, since $\big(\hat{\Phi}\circ{S}\big)_1=\big(a\alpha_{n}u^n,0\big)$, the initial part  $\psi_{\nu}$, $\nu\in\mathbb{N}$  of $\psi$ satisfies the equality $a\alpha_{n}u^n=\psi_{\nu}^{n'}$.
 Then $n=n'\nu$ and therefore $n'\le n$. A similar argument using the inverse diffeomorphism $\Phi^{-1}$ allows us to conclude that $n=n'$, $\nu=1$ and, consequently, $\psi$ has a linear part of the form $\sqrt[n]{a\alpha_n}(x+iy)$. Thus $\psi$ is a formal real diffeomorphism and therefore $\hat{\Phi}\circ{S}$ is a formal real reparametrization of $S'$.
 \end{proof}

\noindent\emph{Proof of Theorem \ref{smooth-equisingular}}. It is a direct consequence of Theorems \ref{smooth-formal} and \ref{formal-equisingular}.

\section{The proof of Theorem \ref{equi-2nd-type-thm}}
\label{theproof}

Let $\F$ and $\F'$ be germs of foliations,   equivalent by a germ of $C^{\infty}$ diffeomorphism
$$\Phi: (\mathbb{C}^{2},0) \to (\mathbb{C}^{2},0).$$
Let $S \in \sep(\F)$ be a branch of separatrix and $S' = \Phi_{*} S \in \sep(\F')$ be the corresponding
separatrix given    by Theorem~\ref{smooth-equisingular}. This result also asserts that, if $\mathscr{S} = \cup_{i=1}^{k}S_i$ is a reduced curve formed by
the union of a finite number of branches in $\sep(\F)$, we set $\mathscr{S}' = \Phi_{*}\mathscr{S} =\cup_{i=1}^{k}\Phi_{*}S_i$, then $\mathscr{S}$ and $\mathscr{S}'$ are equisingular. As a consequence, we have that  $S \in \iso(\F)$ if and only if $S' \in \iso(\F')$ and $S \in \dic(\F)$ if and only if $S' \in \dic(\F')$. We have clearly the following more general fact:

\begin{prop}
\label{s-equising}
$\F$ and $\F'$ are $\cl{S}$-equisingular.
\end{prop}

Suppose now that $\hat{F}$ is a balanced equation
of separatrices for $\F$, whose divisor is  as in \eqref{divisor-bal-eq}. We define
$\hat{F}' = \Phi_{*} \hat{F}$
 as any  formal meromorphic function  corresponding to  the following divisor
\[
 (\hat{F}')_{0}-(\hat{F}')_{\infty} \ = \
\sum_{S\in {\rm Iso}(\F)} (S')+ \sum_{S\in {\rm Dic}(\F)}\ a_{S} (S').
\]

\begin{lem}
\label{dic-correspondence}
Let $\F$ and $\F'$ be  germs of foliations,   equivalent by a germ of $C^{\infty}$ diffeomorphism
$\Phi: (\mathbb{C}^{2},0) \to (\mathbb{C}^{2},0)$.
Let $S \in \dic(\F)$ and $S' = \Phi_{*}S \in \dic(\F') $ be corresponding separatrices,  attached to
 dicritical components $D$ and $D'$ of the desingularizations of $\F$ and $\F'$. Then $\val(D) = \val(D')$.
\end{lem}
\begin{proof} This is a consequence  of Proposition~\ref{s-equising} and of  the following fact:
if $\pi: (M,E) \to (\co^2,0)$ is the reduction of singularities for $\F$ and $\cl{D} \subset E$
is the union of all dicritical components, then each connected component of $E \setminus \cl{D}$ carries
a separatrix of $\F$ (see \cite[Prop. 4]{mol2002}).
\end{proof}

This lemma allows us to prove the following:
\begin{prop}
\label{equiv-balanced-eq}
Let $\F$ and $\F'$ be  germs of foliations,   equivalent by a germ of $C^{\infty}$ diffeomorphism
$\Phi: (\mathbb{C}^{2},0) \to (\mathbb{C}^{2},0)$.
 If $\hat{F}$ is a balanced equations of separatrices for $\F$,
then $\hat{F}' = \Phi_{*} \hat{F}$ is a balanced equation of separatrices for $\F'$. Besides, $\nu_{0}(\hat{F}) = \nu_{0}(\hat{F}')$.
\end{prop}
\begin{proof} The isolated separatrices of $\F$ and $\F'$ are in correspondence by $\Phi$, so that
they appear in the zero divisor of both balanced equations with coefficient 1. Similarly, there is
a correspondence between  dicritical separatrices of $\F$ and $\F'$, which, by
 Lemma~\ref{dic-correspondence}, are attached to dicritical components having the same valences. Therefore, $\hat{F}'$ is a balanced equation
for $\F'$. Finally, the equality on the algebraic multiplicities follows
from the equisingularity property given by Theorem~\ref{smooth-equisingular}.
\end{proof}

The final ingredient for the proof of Theorem \ref{equi-2nd-type-thm} is the following result of \cite{rosas2010}:
\begin{thm}
\label{inv-mult-alg}
Let $\F$ and $\F'$ be germs at $(\co^{n},0)$ of $C^1$ equivalent one dimensional foliations. Then
$\nu_{0}(\F) = \nu_{0}(\F')$.
\end{thm}

This enables to prove the following:

\begin{prop}
\label{tang-excess-invariance}
The tangency excess $\tau_{0}(\F)$ is a $C^{\infty}$ invariant.
\end{prop}
\begin{proof}
Let $\Phi$ be a $C^{\infty}$ equivalence between $\F$ and $\F'$.
We have $\nu_{0}(\F) = \nu_{0}(\F')$ by the previous theorem.
Moreover, Proposition~\ref{equiv-balanced-eq} gives that if $\hat{F}$ be a balanced equation of separatrices for $\F$, then
$\hat{F}' = \Phi_{*}\hat{F}$ is a balanced equation of separatrices for $\F'$ and
$\nu_{0}(\hat{F}) = \nu_{0}(\hat{F}')$.
The result then follows from Proposition~\ref{prop:Equa-Ba}.
\end{proof}

We are now ready to complete the proof of Theorem~\ref{equi-2nd-type-thm}:
\medskip

\noindent {\em Proof  of Theorem~\ref{equi-2nd-type-thm}}. Let  $\F$ and $\F'$ be $C^{\infty}$ equivalent foliations.
Being $\F$ of second type, it holds
 $\tau_{0}(\F) = 0$.  Consequently,  by Proposition~\ref{tang-excess-invariance}, $\tau_{0}(\F') = 0$ and
$\F'$ is also of second type. Hence, both $\F$ and $\F'$ are $\cl{S}$-desingularizable by Proposition~\ref{s-equidesing}.
The proof is accomplished by  using the fact that $C^{\infty}$ equivalent foliations are $\cl{S}$-equisingular
(Proposition~\ref{s-equising}).
\qed

\bibliographystyle{plain}
\bibliography{Biblio}

\medskip \medskip
\noindent
Rog\'erio  Mol  \\
Departamento de Matem\'atica \\
Universidade Federal de Minas Gerais \\
Av. Ant\^onio Carlos, 6627  \  C.P. 702  \\
30123-970  --
Belo Horizonte -- MG,
Brasil \\
rsmol@mat.ufmg.br

\medskip \medskip
\noindent
Rudy Rosas  \\
Pontificia Universidad Cat\'olica del Per\'u \\
Av. Universitaria 1801  \\
Lima,
Peru \\
rudy.rosas@pucp.edu.pe

\end{document}